\documentclass[11pt]{amsart}
 \usepackage{amsmath}
 \usepackage{amssymb}
 \usepackage{amsthm}
\usepackage{color}
\usepackage{epsfig}

\newcommand{\R}{\mathbb R}

 \newtheorem{theo}{Theorem}
 \newtheorem{lem}[theo]{Lemma}

\theoremstyle{definition}

\begin{document}
\title[Phase transition for the geodesic flow]{A remark on the phase transition for the geodesic flow of a rank one surface of nonpositive curvature}
\author{Keith Burns}
\address{Northwestern University, Evanston,  IL 60208, USA}
\email{burns@math.northwestern.edu}
\author{Dong Chen}
\address{The Ohio State University, Columbus, OH 43210, USA}
\email{chen.8022@osu.edu}

%\author{Keith Burns and Dong Chen}
\dedicatory{Dedicated to the memory of Todd Fisher}
\maketitle

\def \R{\mathcal{R}} \def \S{\mathcal{S}} \def \F{\mathcal{F}}

Let $M$ be a compact Riemannian surface of nonpositive curvature  and $\F = (f^t)_{t\in\mathbb{R}}$ is the geodesic flow on the unit tangent bundle $T^1M$.  
A vector in $v \in T^1M$ is called regular if the geodesic to which it is tangent passes through a point where the curvature of $M$ is negative and singular if the curvature of $M$ is $0$ at all points along this geodesic. We denote the sets of regular and singular  vectors  by $\R$ and  $\S$ respectively. It is clear that $\R$ is open and $\S$ is closed.
We assume that both sets are nonempty.

As explained in \cite{BBFS21}, the unstable subbundle $E^u$  is a continuous 
 $\F$-invariant one dimensional subbundle of $TT^1M$ on which the derivative of the geodesic flow is noncontracting. The unstable Jacobian potential  (often called the geometric potential) $\varphi^u: T^1M\to \mathbb{R}$ is defined by
 $$
 \varphi^u(v) = -\lim_{t\to0}\frac{1}{t}\log\det df^t|_{E^u_v} = -\frac{d}{dt}\Big|_{t=0}\log\det df^t|_{E^u_v}.
 $$

  \begin{figure}[htb]
\begin{center}
\includegraphics[width=.6\textwidth]{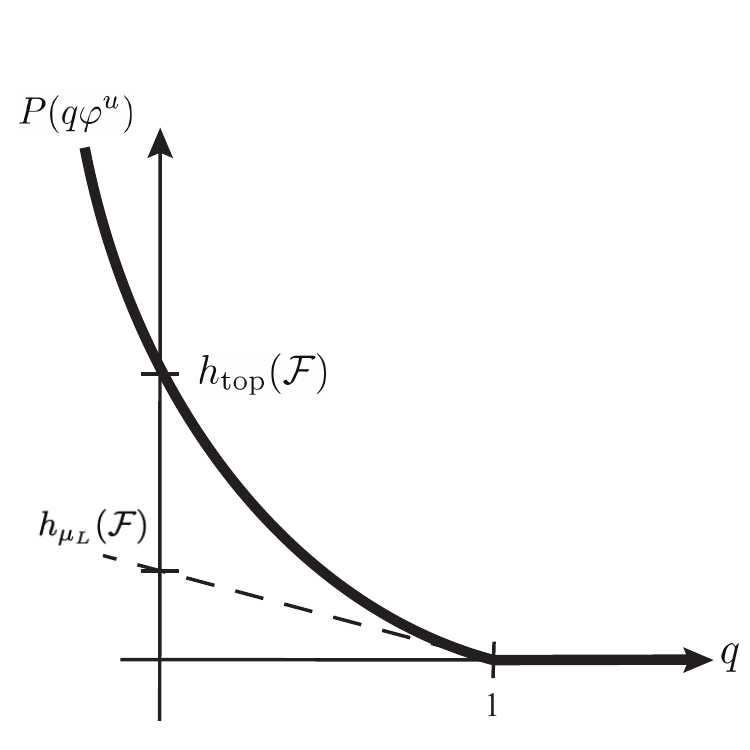}
\caption{Graph of the pressure}\label{f.graph}
\end{center}
\end{figure}

 We denote topological pressure by $P$. The graph of the function $q \mapsto P(q\varphi^u)$ is shown in Figure~\ref{f.graph}, which is copied from \cite{BBFS21}.
 There is  a phase transition at $q = 1$. 
It was shown in \cite{BCFT18} that, for each $q < 1$, there is a unique equilibrium state $\mu_q$ for $q\varphi^u$. Moreover $\mu_q(\R) = 1$ for all $q < 1$. For $q \geq 1$, any measure supported on $\S$ is an equilibrium state for $q\varphi^u$, and these are the only equilibrium states when $q > 1$. The main result of \cite{BBFS21} is that for $\varphi^u$, i.e.\ when $q =1$, there is exactly one additional ergodic equilibrium state $\mu_L$. It is the restriction to $\R$  of the Liouville measure. 

It is natural to ask if $\mu_q \to \mu_L$ as $q \to 1$, but this question was not addressed in \cite{BBFS21}. We answer it here.

\begin{theo}
 $$
 \mu_q \to \mu_L
  $$
 in the weak$^*$-topology as $q \to 1-$.
 \end{theo}
 
 \begin{proof} Our argument is similar to the proof of part 8 of Proposition 5 in \cite{BG14}.  If $\nu$ is a measure on $T^1M$, we set
 $$
 I(\nu) =  \int_{T^1M} \varphi^u(v)\,d\nu(v).
 $$
 
 It suffices to show that if $\mu_{q_n}\to \mu$ in weak$^*$-topology for a sequence $q_n\to 1-$,  then $\mu=\mu_L$. 
 
  It is well known that $I(\mu_L) < 0$. (The opposite inequality is mistakenly asserted in \cite{BBFS21} because  the authors forgot about the minus sign in the definition of $\varphi^u$.)
 Since  $\varphi^u = 0$ on the set $\S$ and any equilibrium state for $\varphi^u$ is a convex linear combination of $\mu_L$ and a measure supported on $\S$, it follows that $I(\nu) > I(\mu_L)$  for any equilibrium state $\nu$ for $\varphi^u$ other than $\mu_L$. It will therefore suffice to show that $\mu$ is an equilibrium state for  $\varphi^u$ and that $I(\mu) \leq I(\mu_L)$.
  
 That $\mu$ is an equilibrium state for $\varphi^u$ follows from  the next lemma, which is a general and well known fact. The entropy map for our our geodesic flow $\F$ is upper semicontinuous, because $\F$  is $h$-expansive (see Proposition 3.3 of \cite{K98}), and it is obvious that the other hypotheses of the lemma hold.
 
 \begin{lem} Let $\F$ be a continuous flow on a compact metric space $X$ for which the entropy map $\nu \mapsto h_\nu(\F)$ is upper semicontinuous. Suppose $\varphi_n: X \to \mathbb R$ are continuous functions that converge uniformly to a continuous function $\varphi: X \to \mathbb R$. For each $n$, let $\nu_n$ be an equilibrium state for $\varphi_n$, and assume that $\nu_n \to \nu$ in the weak-$^*$ topology. Then $\nu$ is an equilibrium state for $\varphi$.
 \end{lem}

 \begin{proof} For each $n$ we have 
  $\displaystyle{
   h_{\nu_n}(\F) + \int_X \varphi_n(x)\,d\nu_n(x) = P(\varphi_n).}
  $
  The hypotheses of the lemma give us
  $$
  \liminf_{n \to \infty}\left[ h_{\nu_n}(\F) + \int_X \varphi_n(x)\,d\nu_n(x) \right] \leq
    h_{\nu}(\F) + \int_X \varphi(x)\,d\nu(x).
  $$
  On the other hand, $P(\varphi_n) \to P(\varphi)$ because the function $P$ is continuous. 
 \end{proof}
 
 We now prove that $I(\mu) \leq I(\mu_L)$.  The following lemma summarizes well-known consequences of the variational principle.
  \begin{lem} 
  Let $\F$ be a continuous flow on a compact metric space $X$ and $\varphi: X\to \mathbb{R}$ a continuous function. Then
  \begin{enumerate}
  \item The function $\mathcal P: \mathbb{R}\to\mathbb{R},  q \mapsto P(q\varphi^u)$ is convex.
  
  \item If $\nu$ is an equilibrium state for $Q\varphi$, then the graph of the function $q \mapsto  h_{\nu} (\F)+ q\int_X \varphi(x) \,d\nu(x)$  is a supporting line for the graph of $\mathcal P$ at $(Q,\mathcal P(Q))$. 
  
  \item If $Q_1\leq Q_2$,  and $\nu_i$ is an equilibrium state for $Q_i\varphi, \, i=1,2$,  then 
  $$\int_X \varphi(x) \,d\nu_1(x)\leq \int_X \varphi(x) \,d\nu_2(x).$$
  \end{enumerate}
 \end{lem}

 %The function $\mathcal P: q \mapsto P(q\varphi^u)$ is convex by the variational principle. The measure  $\mu_{Q}$ is an equilibrium state for $Q\varphi^u$ for each $Q \leq 1$. It follows from the variational principle that,  for $Q \leq 1$, the graph of the function $q \mapsto  h_{\mu_{Q}} (\F)+ qI(\mu_Q)$  is a supporting line for the graph of $\mathcal P$ at the point $(Q,\mathcal P(Q))$. The convexity of the function $\mathcal P$ means that the slopes of these lines cannot decrease as $Q$ increases. 

By the above lemma,  we see immediately that $I(\mu_{q_n}) \leq I(\mu_L)$ for all $n$. Since $I(\mu_{q_n}) \to I(\mu)$ as $n \to \infty$, we obtain $I(\mu) \leq I(\mu_L)$ as desired.
 \end{proof}

%This paper owes much to the work of Todd Fisher in \cite{BCFT18} and \cite{BBFS21}. We shall miss him greatly.

While we were writing this paper,  we learned with sorrow of the recent passing of Todd Fisher.  The work in this paper builds on the work of Todd Fisher in \cite{BCFT18} and \cite{BBFS21},  and we would like to dedicate it to his memory.
 
\textit{Acknowledgment.} The authors would like to thank Federico Rodriguez Hertz for bringing up this question and for suggesting improvements in the exposition.


\begin{thebibliography}{BCFT99}

 \bibitem[BBFS21]{BBFS21}
  Keith Burns, J\'er\^ome Buzzi, Todd Fisher and Noelle Sawyer
  \textit{Phase transitions for the geodesic flow of a rank one surface with nonpositive curvature}.
  \newblock{Dyn.\ Syst.} 36 (2021), no.\ 3, 527--535.  MR4304640
  
\bibitem[BCFT18]{BCFT18}
Keith Burns, Vaughn Climenhaga, Todd Fisher and Daniel J. Thompson
\textit{Unique equilibrium states for geodesic flows in nonpositive curvature}.
\newblock{Geom.\ Funct.\ Anal.} 28 (2018), no. 5, 1209--1259. MR3856792

\bibitem[BG14]{BG14}
Keith Burns and Katrin Gelfert
\textit{Lyapunov spectrum for geodesic flows of rank 1 surfaces}.
\newblock {Discrete Contin.\ Dyn.\ Syst.}
  34 (2014),  no.\ 5, 1841--1872.  MR3124716
  
  \bibitem[K98]{K98}
  Gerhard Knieper
  \textit{The uniqueness of the measure of maximal entropy for geodesic flows on rank $1$ manifolds}.
  \newblock{Ann.\ of Math.} 148, no.\ 1, 291--314. MR1652924
  
 
\end{thebibliography}
\end{document}